\documentclass[11pt,a4paper,amsfonts]{amsart}

\usepackage{amsmath}
\usepackage{fancyhdr}
\usepackage{appendix}
\usepackage{amssymb,amscd,amsxtra,calc}
\usepackage{mathrsfs}
\usepackage{multirow}
\usepackage[all]{xy}
\usepackage{longtable}
\usepackage{enumitem}
\usepackage{mathtools}
\usepackage{comment}
\usepackage{tikz-cd}
\usepackage{array}
\usepackage{booktabs}
\usepackage{stmaryrd}
\usepackage[colorlinks,linkcolor=blue,anchorcolor=blue,citecolor=green,backref=none]{hyperref}

\theoremstyle{plain}
 \newtheorem{thm}{Theorem}[section]
 
 \newtheorem{claim}[thm]{Claim}
 
 \newtheorem{corollary}[thm]{Corollary}
 \newtheorem{lemma}[thm]{Lemma}
 \newtheorem{proposition}[thm]{Proposition}
 
 \newtheorem{theorem}[thm]{Theorem}

\theoremstyle{definition}
 \newtheorem{example}[thm]{Example}
 \newtheorem{definition}[thm]{Definition}
 
 \newtheorem*{notation*}{Notation and Terminology}
 \newtheorem{remark}[thm]{Remark}
\theoremstyle{remark}

\newcommand{\diag}{\operatorname{diag}}

\newcommand{\ed}{\operatorname{ed}}

\newcommand{\Fix}{\operatorname{Fix}}
\newcommand{\Gal}{\operatorname{Gal}}
\newcommand{\GL}{\operatorname{GL}}

\newcommand{\id}{\operatorname{id}}

\newcommand{\Ker}{\operatorname{Ker}}

\newcommand{\PGL}{\mathrm{PGL}}
\newcommand{\Proj}{\operatorname{Proj}}

\newcommand{\rank}{\operatorname{rank}}

\newcommand{\Alb}{\operatorname{Alb}}

\newcommand{\Pic}{\operatorname{Pic}}

\usepackage[colorinlistoftodos]{todonotes}

\newcommand{\nc}{\newcommand}
\nc{\cH}{{\mathcal H}}
\nc{\cA}{{\mathcal A}}
\nc{\cG}{{\mathcal G}}
\nc{\cC}{{\mathcal C}}
\nc{\cO}{{\mathcal O}}
\nc{\cI}{{\mathcal I}}
\nc{\cB}{{\mathcal B}}
\nc{\cY}{{\mathcal Y}}
\nc{\cK}{{\mathcal K}}
\nc{\cX}{{\mathcal X}}
\nc{\cS}{{\mathcal S}}
\nc{\cE}{{\mathcal E}}
\nc{\cF}{{\mathcal F}}
\nc{\cZ}{{\mathcal Z}}
\nc{\cQ}{{\mathcal Q}}
\nc{\cN}{{\mathcal N}}
\nc{\cP}{{\mathcal P}}
\nc{\cL}{{\mathcal L}}
\nc{\cM}{{\mathcal M}}
\nc{\cT}{{\mathcal T}}
\nc{\cW}{{\mathcal W}}
\nc{\cU}{{\mathcal U}}
\nc{\cJ}{{\mathcal J}}
\nc{\cV}{{\mathcal V}}

\nc{\bH}{{\mathbb H}}
\nc{\bA}{{\mathbb A}}
\nc{\bG}{{\mathbb G}}
\nc{\bC}{{\mathbb C}}
\nc{\bO}{{\mathbb O}}
\nc{\bI}{{\mathbb I}}
\nc{\bB}{{\mathbb B}}
\nc{\bY}{{\mathbb Y}}
\nc{\bK}{{\mathbb K}}
\nc{\bX}{{\mathbb X}}
\nc{\bS}{{\mathbb S}}
\nc{\bE}{{\mathbb E}}
\nc{\bF}{{\mathbb F}}
\nc{\bZ}{{\mathbb Z}}
\nc{\bQ}{{\mathbb Q}}
\nc{\bN}{{\mathbb N}}
\nc{\bP}{{\mathbb P}}
\nc{\bL}{{\mathbb L}}
\nc{\bM}{{\mathbb M}}
\nc{\bT}{{\mathbb T}}
\nc{\bW}{{\mathbb W}}
\nc{\bU}{{\mathbb U}}
\nc{\bD}{{\mathbb D}}
\nc{\bJ}{{\mathbb J}}
\nc{\bV}{{\mathbb V}}
\nc{\bbZ}{{\mathbb Z}}
\nc{\bR}{{\mathbb R}}
\nc{\fr}{{\rightarrow}}
\nc{\co}{{\nabla}}

\nc{\cu}{{\overlineline{\nabla}}}

\title [Essential dimensions]{Essential dimensions of polarized endomorphisms of certain algebraic surfaces}

\author{Qianli Fang, Yujie Luo and De-Qi Zhang}
\date{}

\address{Department of Mathematics, National University of Singapore, Singapore 119076,
Singapore}

\email{fqianli@u.nus.edu}

\address{Department of Mathematics, National University of Singapore, Singapore 119076, 
Singapore}

\email{lyj96@nus.edu.sg}

\address{Department of Mathematics, National University of Singapore, Singapore 119076, 
Singapore}

\email{matzdq@nus.edu.sg}


\begin{document}

\begin{abstract}
Let $f: X \to X$ be a polarized endomorphism of a smooth projective surface which is birationally ruled. We answer a question of Koll\'ar and Zhuang, in the affirmative, on the incompressibility of $f$, under the assumption that $f$ is Galois and an explicit lower bound of $\deg f$ depending only on $X$. We also give examples showing the optimality of such a lower bound.
\end{abstract}

\maketitle

\tableofcontents


%
%
%
%
\section{Introduction}

Throughout this article, we work over an algebraically closed field {\bf k} of characteristic zero.

\medskip

The essential dimension was first introduced in a modern context by Buhler and Reichstein \cite{BR97, BR99}. It has been studied for decades. We refer the readers to Reichstein's ICM talk for more on its history \cite{Reic10}.  To be precise, the \emph{essential dimension of a generically finite (dominant rational) map} $f: X \dasharrow Y$, denoted by $\mathrm{ed}(f)$, is the smallest integer
$d$ such that $f$ is birational to the pull-back of a map of varieties of dimension $d$ (cf. \cite{BR97, FKW24, KZ}). A generically finite map $f: X\dasharrow Y$ is called \emph{incompressible} if its essential dimension is equal to its maximum possible value $\dim(X)$ \cite{BR97,FKW24,KZ}.

Recall that an endomorphism $f$ of a projective variety $X$ is {\it polarized} if there is an ample divisor $H$ such that $f^*H\sim qH$ for some positive integer $q>1$. In this case, $\deg f = q^{\dim X}$ by the projection formula. Koll\'ar and Zhuang proved a conjecture of Brosnan (cf. \cite[Conjecture~6.1]{FS22}) asserting that the multiplication-by-$m$ map $m_A: A \to A$ of an abelian variety $A$ is incompressible for $m\geq 2$ (cf.~\cite[Theorem~1]{KZ}), and they further asked \emph{whether every polarized endomorphism is incompressible} \cite[Question~19]{KZ}. The  assumption that $f$ is polarized is necessary as there are examples of so-called \emph{int-amplified} endomorphisms $f$ which are not incompressible; see \cite[\S~4]{KZ}, and also Luo, Oguiso and Zhang \cite[Example~5.12]{LOZ25}.

Let $X$ be a normal projective variety and let $f$ be a polarized endomorphism of $X$. Then we have the Kodaira dimension $\kappa(X)\leq 0$ by Nakayama and Zhang \cite[Theorem~1.3]{NZ}. If the Kodaira dimension $\kappa(X) = 0$, then $f$, possibly replaced by an iteration, lifts to a polarized endomorphism $f_A: A \to A$ of an abelian variety $A$ (cf. \cite[Theorem~3.4]{NZ} and \cite[Theorem~1.21]{GKP13}). In \cite{LOZ25}, the authors have studied the essential dimension of polarized endomorphisms of abelian varieties. In dimension two, it remains for us to study polarized endomorphisms of smooth projective surfaces $X$ which are birationally ruled, i.e., the Kodaira dimension $\kappa(X) = -\infty$. Our first result deals with rational surfaces.

We say that $f: X = X_1 \to X = X_2$ is {\it Galois} if the function field extension $K(X_1)/K(X_2)$ is Galois with Galois group $G$ and $G$ acts biregularly on $X_1$ (and hence $X_2 = X_1/G$).

\begin{theorem}\label{Thm: main}
Let $X$ be a smooth projective rational surface, and let $f: X \to X$ be a $q$-polarized endomorphism such that $q > 2$. Then the following assertions hold.
\begin{enumerate}
 \item
 Suppose $X$ is not isomorphic to $\mathbf{P}^1 \times \mathbf{P}^1$, and $f$ is Galois. Then $f$ is incompressible.
 \item
 Suppose $X=\mathbf{P}^1 \times \mathbf{P}^1$, $f$ is Galois and $f$ preserves both fibrations of $X$ (this holds for $f^2$ if $f^2$ is Galois). Then $f$ is incompressible.
\end{enumerate}
\end{theorem}

\begin{remark}\label{rem: Thm and Cor}
By Example \ref{preserve projection}, in Theorem \ref{Thm: main}(2), the condition that $f$ preserves both fibrations when $X = \mathbf{P}^1 \times \mathbf{P}^1$, is necessary.
\end{remark}

Suppose $X$ is a smooth projective surface and $f: X \to X$ is a polarized endomorphism. If $X$ is birationally ruled and irrational, then $X$ is a relatively minimal ruled surface over an elliptic curve  
(see \cite[Theorem~1]{Nak02}). In fact, S.~W.~Zhang has a complete classification of algebraic surfaces that admit polarized endomorphisms \cite[Proposition~2.3.1]{ZhS06}.

\begin{theorem}\label{thm: elliptic ruled}
 Let $X$ be an irrational birationally ruled smooth surface, and $f$ a $q$-polarized endomorphism of $X$ with $q>2$. We may write $X=\mathrm{Proj}_C(V)$, with $C$ an elliptic curve and $V$ a rank two vector bundle on $C$. Assume that $f$ is Galois. Then the following assertions hold.
 \begin{enumerate}
 \item Suppose $V$ is indecomposable. Then $f$ is incompressible.
 \item Suppose $V$ is decomposable. Then there exists a positive integer $N$ depending only on $X$ such that if $q>N$ (i.e., $\deg f > N^2$), $f$ is incompressible.
 \end{enumerate}
\end{theorem}

\begin{remark}\label{rem: low bd of deg for ed}
By the proof, we have an optimal choice of the integer $N$ in Theorem~\ref{thm: elliptic ruled}(2). In fact, we may write the decomposable $V$ as $\mathcal{O}_C\oplus \mathcal{L}$. When $\deg(\mathcal{L})\neq 0$, we may take $N:=\max\{2,  {|\deg(\mathcal{L})|}\}$. When $\deg(\mathcal{L})=0$, $\mathcal{L}$ is a torsion line bundle by \cite[Proposition~2.3.1]{ZhS06}, and we may take $N:=\max\{2,  {\rm ord}(\mathcal{L})\}$.
For both cases, the values for $N$ are optimal; see Examples~\ref{ex: positive degree q=degree} and \ref{ex: torsion optimality}.
\end{remark}

Here is the organization of the paper. At the end of Section \ref{sect: rat surf}, we prove Theorem \ref{Thm: main}. In Section \ref{sect: irrat surf}, after preparing some lemmas, we prove Theorem \ref{thm: elliptic ruled} and Remark \ref{rem: low bd of deg for ed}.

\vspace{2mm}

{\bf Acknowledgements.}
The authors are supported, respectively, by an NUS graduate scholarship, the Peng Tsu Ann Assistant Professorship, and the ARF: A-8002487-00-00 of NUS.
\medskip

\section{Preliminary results}
We adopt the standard notation from \cite{Ful}, \cite{Har}, \cite{Laz}, \cite{KM} and \cite{Mum74}. We begin with some notation for the ramification divisors and group actions.

\subsection{Ramification divisors and group actions}

\begin{definition}\label{defn: fixed locus pointwise}
(1)
Let $f: X \to Y$ be a finite surjective morphism between $\mathbb{Q}$-factorial normal varieties. For canonical divisors (regarded as integral Weil divisors), we have the {\it ramification divisor formula}:
$$K_X = f^*K_Y + R_f$$
where $R_f$ is an effective Weil integral divisor, called the {\it ramification divisor} of $f$. Write $R_f = \sum (r_i-1) R_i$ for some distinct prime divisors $R_i$ and some positive integers $r_i$. Then $r_i$ is called the {\it ramification index} of $f$ along $R_i$. 

(2)
Let $G$ be a finite group acting faithfully on a variety $X$. For each element $g\in G$, define the \emph{(pointwise) fixed locus} of $g$ as
$$\Fix(g):= \{x \in X \mid g(x) = x\}.$$ Note that $\Fix(g)$ is Zariski closed in $X$. Let $f: X\to X/G$ be the quotient morphism which is finite. Then the ramification divisor $R_f$ has support equal to the union of codimension one components $R_i$ of $\mathrm{Fix}(G)$. Moreover, we define the (pointwise) {\it stabilizer} as
$$G_{R_i} = \{g\in G\mid g(x)=x \text{ for each }x\in R_i\}.$$ Then the ramification index $r_i$ of $f$ along $R_i$ is equal to $|G_{R_i}|$. In particular, $r_i$ divides $|G|$.
\end{definition}

\begin{definition}
    Let $V$ be a finite-dimensional vector space over $\bf k$. A nontrivial element $\sigma\in \GL(V)$ of finite order is called a {\it pseudo-reflection} if it fixes a hyperplane in $V$ pointwise. A finite subgroup $\Gamma\subset \GL(V)$ is called a {\it pseudo-reflection group} if it is generated by pseudo-reflections. 
\end{definition}

\begin{theorem}(\cite[Corollary~1.8]{Sat12})\label{lem: generated by pref}
    Let $X$ be a smooth affine variety with a faithful action by a finite group $G$. Let $P$ be a point in $X$ fixed by $G$. Then $G$ is generated by pseudo-reflections if and only if $X/G$ is smooth at the image of $P$.
\end{theorem}

We frequently use the following lemma due to H. Kraft and I. Stampfli \cite{KS13} whose proof works for any quasi-projective variety.

\begin{lemma}(\cite[Lemma~2.2]{KS13})\label{lem: reductive faithful act}
Let G act faithfully on a quasi-projective variety $X$. Assume that $x_0\in X$ is a fixed closed point and that there is a $G$-stable decomposition (for the maximal ideal of the local ring at $x_0$) $$(*) \hskip 1pc m_{x_0}=V\oplus m_{x_0}^2.$$ Then the tangent representation $\tau:G\to \GL(T_{x_0}X)$ is faithful. In particular, a G-stable decomposition $m_{x_0} = V\oplus m_{x_0}^2$ as in the $(*)$ above always exists if G is a finite or, more generally, a reductive algebraic group.
\end{lemma}

\begin{lemma}\label{lem: being cyc or dih}
Let $h: \mathbf{P}^1 \to \mathbf{P}^1$ be a Galois cover with Galois group $G$. Suppose $\deg(h)$ is a perfect square. Then the following assertions hold.
\begin{enumerate}
\item[(1)]
$G$ is either a cyclic group or a dihedral group $D_{2n}$ of order $2n$ ($n \ge 2$).
\item[(2)]
Suppose that $G = D_{2n}$, and write
$$G = \langle \tau, \sigma \mid \tau^n = 1, \sigma^2 = 1, \sigma \tau \sigma^{-1} = \tau^{-1} \rangle.$$ Then we have:
\item[(2i)]
The abelianization $G^{ab}:=G/[G,G]$ is $\mathbb{Z}/2\mathbb{Z}$ if $n$ is odd, and is $\mathbb{Z}/2\mathbb{Z} \times \mathbb{Z}/2\mathbb{Z}$ if $n$ is even. 
\item[(2ii)]
The only abelian subgroups of $G$ are $\langle \tau^r \rangle$ for some $r$, or  $\langle \sigma \tau^r \rangle \cong {\mathbb Z}/2{\mathbb Z}$ for some $r$, or $\langle \sigma \tau^r, \tau^{n/2} \rangle \cong \mathbb{Z}/2\mathbb{Z} \times \mathbb{Z}/2\mathbb{Z}$ (when $n$ is even, for some $r$). 
\end{enumerate}
  
\end{lemma}

\begin{proof}
(1) We note that $G$ is isomorphic to a finite subgroup of $\mathrm{PGL}(2, \mathbf{k})$. By the classification of finite subgroups of $\mathrm{PGL}(2, \mathbf{k})$ \cite[Introduction]{Bea10}, $G$ is cyclic, dihedral or isomorphic to $A_4$, $S_4$, or $A_5$. Since $|G|$ is a perfect square, the cases $A_4$, $S_4$, $A_5$, whose orders are $12$, $24$, $60$, respectively, are excluded. Hence $G$ is either a cyclic group or a dihedral group.

(2) Now, suppose that $G \cong D_{2n}$. We may write $G$ as
$$G = \langle \tau, \sigma \mid \tau^n = 1, \sigma^2 = 1, \sigma \tau \sigma^{-1} = \tau^{-1} \rangle.$$ In particular, the commutator subgroup $[G, G]=\langle [\tau, \sigma]\rangle=\langle \tau^2\rangle$. When $n$ is odd, $[G,G]=\langle \tau\rangle$, hence $G^{ab} = \langle \bar{\sigma}\rangle \cong \mathbb{Z}/2\mathbb{Z}$. When $n$ is even, $G^{ab} = \langle \bar{\sigma}, \bar{\tau}\rangle \cong \mathbb{Z}/2\mathbb{Z} \times \mathbb{Z}/2\mathbb{Z}$. This proves (2i), while (2ii) follows from the description of $G$.
\end{proof}

\begin{lemma}\label{lem: stabilizer-abelian-at-node}
Let $G$ be a finite group acting faithfully on a smooth surface $X$. Suppose that there are two $G$-stable smooth curves $C_1, C_2$ such that $\{p\} = C_1 \cap C_2$ and they intersect transversely at $p$. Then $G$ is abelian.
\end{lemma}

\begin{proof}
By our assumption, $p$ is a fixed point of $G$. We may choose local coordinates $x,y$ at $p$ such that locally
$C_1=\{x=0\}, \, C_2=\{y=0\}.$

Since every element of $G$ preserves both $C_j$, it preserves the two tangent lines $T_pC_1$ and $T_pC_2$ in $T_pX$. Hence the tangent representation (which is faithful by Lemma \ref{lem: reductive faithful act}) $$\rho\colon G\longrightarrow \GL(T_pX)$$ splits as a direct sum of the representations $\rho_i: G\to \mathrm{GL}(T_p C_i)$ ($i=1,2$). In particular, all elements of $\rho(G)$ are simultaneously diagonalizable. It follows that $G$ is abelian.
\end{proof}

\subsection{Essential dimensions}

We start with the definition of the essential dimension of a generically finite dominant rational map.

\begin{definition}[cf. {\cite[Definition~2.1]{BR97}}]
 Let $E/F$ be a finite field extension over a base field $\mathbf{k}$. We say that $E/F$ is \emph{defined over} a subfield $F_0$ of $F$ if there exists an extension $E_0/F_0$ over $\mathbf{k}$ such that $[E_0:F_0]=[E:F]$, $E_0 \subseteq E$ and $E_0F = E$. The \emph{essential dimension} of $E/F$, which we will usually abbreviate as $\mathrm{ed}_\mathbf{k}(E/F)$ (or $\mathrm{ed}(E/F)$ when $\mathbf{k}$ is clear from the context), is the minimal value of $\mathrm{trdeg}_\mathbf{k}(F_0)$ as $F_0$ ranges over all fields for which $E/F$ is defined over $F_0$. The \emph{essential dimension of a generically finite dominant rational map} $f: X\dasharrow Y$, denoted by $\mathrm{ed}(f)$, is defined as the essential dimension of the corresponding function field extension $\mathrm{ed}_\mathbf{k}(K(X)/K(Y))$.
\end{definition}

We collect the following facts on essential dimensions defined above (see \cite[Lemma~2.2]{LOZ25}).

\begin{lemma}\label{lem: facts of essential dimension}
We fix the base field $\mathbf{k}$ of characteristic zero.
\begin{enumerate}
 \item If two generically finite dominant rational maps are birationally equivalent, then they have the same essential dimension. In particular, the essential dimension of a birational map is zero.
 \item Let $f: X\dasharrow Y$ and $g: Y\dasharrow Z$ be generically finite dominant rational maps. Then $\max\{\mathrm{ed}(g),\mathrm{ed}(f)\}\leq \mathrm{ed} (g\circ f)$. In particular, $\mathrm{ed}(g_2\circ f\circ g_1)=\mathrm{ed}(f)$, when $g_1$ (resp. $g_2$) is a birational map on $X$ (resp. $Y$).
 \item Let $f: X \dasharrow Y$ be a generically finite dominant rational map such that the induced field extension $K(X)/K(Y)$ is abelian Galois with Galois group $G$. Then $\mathrm{ed}(f) \leq \rank(G)$.
\end{enumerate}
\end{lemma}

\subsection{Algebraic surfaces with polarized endomorphisms}

Polarized endomorphisms are key objects in algebraic dynamics. We collect some equivalent descriptions of polarized endomorphisms of normal varieties.

\begin{definition}\label{defn: q-polarized}
 Let $f: X\to X$ be a finite surjective endomorphism of a normal projective variety $X$. We say $f$ is \emph{$q$-polarized} or simply \emph{polarized} if there exists an ample Cartier divisor (or equivalently, an ample $\mathbb{R}$-Cartier divisor \cite[Lemma~3.5]{MZ18}) $H$ on $X$ such that $f^* H\sim qH$ (or equivalently, $f^*H\equiv qH$ (numerical equivalence) \cite[Lemma~2.3]{NZ}) for some integer (or equivalently, some rational number \cite[Lemma~3.5]{MZ18}) $q > 1$.

 The above notions behave well under iterations. To be more specific, if $f^s$ is $q^s$-polarized for some integer $q > 1$, then $f$ is $q$-polarized; see \cite[Proof of Theorem~2.7, Note 1]{Zh10}.
\end{definition}

Below is the classification theorem for projective surfaces admitting polarized endomorphisms.

\begin{thm}[{\cite[Proposition~2.3.1]{ZhS06}}]\label{thm: classification of dynamical surface}
 Let \(\phi : X \to X\) be a polarized endomorphism of a smooth projective surface. Then \(X\) is one of the following types:
\begin{enumerate}
 \item[(1)] Abelian surface.
 \item[(2)] Hyperelliptic surface, i.e., the \'etale quotient of the product of two elliptic curves.
 \item[(3)] Toric surface.
 \item[(4)] Ruled surface \(\mathbf{P}_C(\mathcal{E})\) over an elliptic curve such that either
 \begin{itemize}
 \item[(4i)] $\mathcal{E} = \mathcal{O}_C \oplus \mathcal{M}$ with $\mathcal{M}$ torsion or of positive degree, or
 \item[(4ii)] $\mathcal{E}$ is not decomposable and of odd degree.
 \end{itemize}
\end{enumerate}
\end{thm}

For the reader's convenience, we attach the proof of the following simple lemma, which can be viewed as a corollary of Theorem~\ref{thm: classification of dynamical surface}.

\begin{lemma}\label{lem: being ell ruled}
Let $X$ be a smooth projective surface that is birationally ruled and irrational, and let $f: X \to X$ be a $q$-polarized endomorphism. Then $X$ is a (relatively minimal) ruled surface over an elliptic curve $B$, and $f$ descends to a $q$-polarized endomorphism of $B$.
\end{lemma}

\begin{proof} 
By Theorem \ref{thm: classification of dynamical surface}, $X$ is a ruled surface over an elliptic curve $B$. The Albanese morphism of $X$ is precisely the ruling. Therefore, any surjective endomorphism $f:X\to X$ descends to an endomorphism of $B$. By \cite[Theorem 1.3]{MZ18}, the induced map on $B$ is $q$-polarized.
\end{proof}

\subsection{Negative curves on polarized surfaces}

A curve $C$ on a smooth projective surface $X$ is called a {\it negative} curve if the self-intersection number $C^2 < 0$.

\begin{lemma}\label{lem: finite neg curve}
Let $X$ be a smooth projective surface and let $f: X \to X$ be a surjective endomorphism of $\deg f \ge 2$. Then:
\begin{enumerate}
\item[(1)]
$f^{-1}$ takes every negative curve on $X$ to an irreducible negative curve. Hence both $f$ and $f^{-1}$ induce bijections on the set of negative curves $C$ of $X$ (which is a finite set). In particular, we have $f^{-1}(f(C)) = C$, $f^*f(C) = r C$ and $f_*C = (\deg (f)/r )f(C)$; here
$r$ is the ramification index of $f$ along $C$.
\item[(2)]
Suppose further that $f$ is Galois with Galois group $G$. Then every negative curve $C$ on $X$ is $G$-stable, and the $r$ in (1) is also equal to the order of $G_C =\{g \in G \, | \, g|_C = \id\}$.
\end{enumerate}
\end{lemma}

\begin{proof}
(1) follows from Nakayama \cite[Proposition~10]{Nak02} and the projection formula.

For (2), since $f$ is Galois with Galois group $G$, we have that $f^{-1}(C)=\cup_{g\in G} g(C')$, where $C'$ is an irreducible curve such that $f(C')=C$. In particular, $g(C')=C'$ for all $g\in G$ by (1). Since $C'$ ranges over all negative curves as $C$ varies (by (1)), every negative curve on $X$ is $G$-stable. For the remaining assertion, note that $\deg(f|_{C})=|G/G_C|=\deg(f)/r$ by (1). Since $\deg(f)=|G|$, (2) follows.
\end{proof}

\begin{lemma}\label{lem: 2-Go}
Let $X$ be a smooth projective rational surface, and $f:X\to X$ be a cyclic Galois self-cover with the Galois group $G=\langle g\rangle$ such that $d:= \deg f = |G| \ge 2$. Denote by $D$ the reduced divisor
which is the sum of all the irreducible curves $C_i$ on $X$ with negative self-intersection. Assume $D \ne 0$, $D$ is connected, each $C_i \cong \mathbf{P}^1$,
and $C_i . C_j \le 1$ for all $i \ne j$. Then the following statements hold.
\begin{enumerate}
 \item[(1)]
 $D$ is of simple normal crossing.
 \item[(2)] 
 Suppose that the number $\#D$ of irreducible components of $D$ is less than or equal to $3$. Then $\#D  \in \{1, 3\}$, $D$ is a linear chain and there is an irreducible component $C$ of $D$ such that $f^{-1}(C)=C$.
 \item[(3)]
 Suppose that $\#D$ is larger than or equal to $4$. 
 Then $\mathrm{deg}(f) \le 3$.
\end{enumerate}
\end{lemma}
\begin{proof}
Write $D=\sum_{i=1}^s C_i$, where $C_i$ are irreducible components of $D$ . If $s = \#D = 1$, the lemma is clear. If $s \ge 2$ then $X$ is the blowup of some Hirzebruch surface and hence contains at least three negative curves. Thus we may assume that $s \ge 3$. Hence $D$ contains a $(-1)$-curve.

Let $K_i:=\{g^r\in G:{g^r}|_{C_i}=\id_{C_i}\}$ be the pointwise stabilizer subgroups of negative curves. Let $\{P\} = C_i \cap C_{j}$. Since every negative curve is $G$-stable by Lemma \ref{lem: finite neg curve}, $P$ is fixed by $G$.
After choosing the local coordinates $x,y$ at $P$ such that locally $C_i=\{x=0\}$ and $C_{j}=\{y=0\}$, we may write the action of $g$ as $g(x,y)=(\lambda x,\mu y)$, where $\lambda, \mu$ are powers of a primitive $\deg(f)$-th root of unity $\zeta$. Hence, for some integer $r$, $g^r\in K_i$ if and only if $\mu^r=1$, and $g^r\in K_{j}$ if and only if $\lambda ^r=1$. By Lemma \ref{lem: reductive faithful act}, the representation $G \to G|_{{T_{P}}X}$ is an isomorphism. Hence, $K_i\cap K_{j}=1$. Since $P$ is fixed by $G$, and taking a $G$-stable affine neighborhood of $P$ and by Lemma \ref{lem: generated by pref}, $G$ is generated by pseudo-reflections on $T_{P}X$. Hence, $G = G_{P}=\langle K_i,K_{j}\rangle$. Since $G$ is cyclic, we have $G=K_i K_{j}$, and hence $d=|K_i|\cdot|K_{j}|$, with $\gcd(|K_i|,|K_{j}|)=1$. Label each vertex of the dual graph of $D$ by the ramification index of the corresponding irreducible component. Then the product of the two labels at the endpoints of any edge is $d$. By Lemma \ref{lem: finite neg curve}, $f^*f(C_i) = |K_i|C_i$. 
Since $D$ contains a $(-1)$-curve, we may assume $C_i^2=-1$ for some $i$. Then by the projection formula, $(f^*f(C_i))^2=df(C_i)^2=-|K_i|^2$. Hence 
\begin{equation}\label{eq: alteration of index}
    \hskip 1pc |K_i|=d, \, |K_{j}| = 1.
\end{equation}
 Starting from the vertex corresponding to $C_i$, whose ramification index is $d$, and using the connectedness of the dual graph of $D$, we see by induction along the edges that every component of $D$ has ramification index either $1$ or $d$, and every two intersecting curves $C_i$ and $C_j$ of $D$ have pointwise stabilizer subgroups satisfying the equalities \eqref{eq: alteration of index} above. This and $d \ge 2$ imply that no three $C_i$ of $D$ share a common point, so the assertion (1) is true.

Suppose that $D=C_1+C_2+C_3$.
If $D=C_1+C_2+C_3$ is a simple loop, then the vertex with label $d$ is adjacent to the other two vertices, so both of them have label $1$. Since these two vertices are also adjacent, the product rule gives $d=1\cdot 1=1$, contradicting $d\ge 2$. If $D$ is a linear chain, then $f^{-1}(C_2)=C_2$. This proves the assertion (2).

For $C_i$ with $|K_i|=1$, by Lemma \ref{lem: finite neg curve}, $\deg(f|_{C_i})=\deg f$ and $f^*f(C_i) = C_i$. By the projection formula,
\begin{equation}\label{eq: upper bound for self-intersection}
    C_i^2=(f^*f(C_i))^2=\mathrm{deg}(f) (f(C_i))^2\leq -\mathrm{deg}(f).
\end{equation}

For the assertion (3), suppose $\deg(f) \ge 4$ and $s\ge4$, we will deduce a contradiction. Indeed, we first claim that $D$ has at least two irreducible components with ramification index $1$ and hence self-intersection numbers at most $-d$ by the inequality (\ref{eq: upper bound for self-intersection}). In fact, if $D$ is a linear chain, then the claim follows from the equalities \eqref{eq: alteration of index}. Suppose some irreducible component $C_1$ ($\cong {\bf P}^1$) of $D$ meets at least three other irreducible components $C_j$ of $D$. Then $g|_{C_1}$ has at least three fixed points, so $g|_{C_1} = \id$. In particular, the ramification index of $f$ along $C_1$ is $d$. Thus the ramification index of each $C_j$ meeting $C_1$ is equal to $1$ by the equalities \eqref{eq: alteration of index}, so the claim holds in this case too.

Let $h: X \to X'$ be the composition of  blow-downs of $(-1)$-curves so that $X'$ is relatively minimal. For every $C_j$ in $D$ with ramification index $1$ (there are two such $C_j$ by the claim above and each $C_j$ meets at most two others $C_k$, or else $g|_{C_j} = \id$, absurd!), we have$$h(C_j)^2\le C_j^2+2\le -\deg f+2\le -2.$$
Here the first inequality above becomes an equality only when two $(-1)$-curves meeting $C_j$ are contracted by $h$, while the second inequality follows from the inequality \eqref{eq: upper bound for self-intersection}.
 Hence $X'$ would contain at least two negative curves, which is absurd. This proves the assertion (3).
\end{proof}

\begin{lemma}\label{lem: loop criterion}
Let $X$ be a smooth projective rational surface admitting a $q$-polarized endomorphism $f$, and let $0 \ne D=\sum_{i=1}^n D_i$ be a reduced divisor. Let $\Gamma(D)$ be the dual graph of $D$. Then the following assertions hold.
\begin{enumerate}
 \item $(\dim H^0(X, K_X + D) =) \, h^0(X,K_X+D)>0$ if and only if either some $D_i$ satisfies $p_a(D_i)\ge 1$ or $\Gamma(D)$ contains a loop.
 \item Suppose further that $h^0(X,K_X+D)>0$ and each $D_i$ is $f^{-1}$-periodic. Then $K_X + D \sim 0$. Either $D=D_1$ and the arithmetic genus $p_a(D_1)=1$, or $D$ is a simple loop of smooth rational curves.
\end{enumerate}
\end{lemma}

\begin{proof}
(1) Since $X$ is rational, $\chi(\mathcal{O}_X)=1$ and $h^1(\mathcal{O}_X)=h^2(\mathcal{O}_X)=0$. By the Riemann-Roch theorem, we have $$\chi(\mathcal{O}_X(K_X+D))=\frac{1}{2}(K_X+D)\cdot D+1.$$
By Serre duality, $h^2(X,K_X+D)=h^0(X,-D)=0$ and $h^1(X,K_X+D)=h^1(X,-D)$. Hence $$h^0(X,K_X+D)=\chi(\mathcal{O}_X(K_X+D))+h^1(X,-D).$$

Consider the short exact sequence $$0\to \mathcal{O}_X(-D)\to \mathcal{O}_X\to \mathcal{O}_D\to 0.$$ It induces an exact sequence $$0=h^0(X,\mathcal{O}_X(-D))\to h^0(X,\mathcal{O}_X)\to h^0(X,\mathcal{O}_D)\to h^1(X,\mathcal{O}_X(-D)) \to h^1(X,\mathcal{O}_X)=0.$$
If $D$ has $t$ connected components, then $h^0(\mathcal{O}_D)=t$. It follows that $h^1(X,\mathcal{O}_X(-D))=t-1.$
Therefore
\begin{equation}\label{eq: section for the log canonical divisor}
 h^0(X,K_X+D)=\frac{1}{2}(K_X+D)\cdot D+t.
\end{equation}
Now we compute
\begin{equation}\label{eq: adjunction formula}
 (K_X+D)\cdot D=\sum_i (K_X+D_i)\cdot D_i+2\sum_{i<j} D_i\cdot D_j=\sum_i( 2p_a(D_i)-2)+2\sum_{i<j} D_i\cdot D_j,
\end{equation}
where the last equality is by the adjunction formula. Combining equations \eqref{eq: section for the log canonical divisor} and \eqref{eq: adjunction formula}, we have
\begin{equation}\label{eq: section formula final}
 h^0(X,K_X+D)=\sum_i p_a(D_i)-n+\sum_{i<j} D_i\cdot D_j+t=\sum_i p_a(D_i)+ b_1(\Gamma(D)),
\end{equation}
where $\sum_{i<j} D_i\cdot D_j$ denotes the number of edges of the dual graph $\Gamma(D)$ and $b_1(\Gamma(D))$ is the first Betti number of $\Gamma(D)$. The assertion (1) follows immediately from the equality \eqref{eq: section formula final}.

\medskip

(2) After replacing $f$ by a suitable iterate, we may assume that $f^{-1}(D_i)=D_i$ for all $i$. Let $D'\subseteq D$ be a reduced subdivisor such that $h^0(X,K_X+D')>0$. By the log ramification formula, $$K_X+D'=f^*(K_X+D')+R_f',$$ where $R_{f'}\geq 0$ (here we used the fact that $f^{-1}(D_i)=D_i$ for all $i$ ).
By iterating $f$, we obtain
\begin{equation}\label{eq: iteration and pull back log canonical divisor}
 K_X+D'=(f^*)^n(K_X+D')+\sum_{i=0}^{n-1}(f^*)^i(R_f').
\end{equation}
We choose an ample divisor $H$ such that $f^*H\equiv qH$. Intersecting the equality \eqref{eq: iteration and pull back log canonical divisor} with $H$ on both sides, we get the following equality
\begin{equation}\label{eq: intersection formula for iterated pull back of log pair}
 H\cdot (K_X+D')=H\cdot (f^*)^n(K_X+D')+\sum_{i=0}^{n-1}H\cdot (f^*)^i(R_f')=q^n H\cdot (K_X+D')+(\sum_{i=0}^{n-1} q^i)H\cdot R_f',
\end{equation}
where the last equality is by the projection formula. Since the left-hand side of the equality \eqref{eq: intersection formula for iterated pull back of log pair} is bounded when $n\to \infty$ and the first term in the right-hand side is non-negative, the ampleness of $H$ implies $H\cdot (K_X+D')=H\cdot R_f'=0$. Hence $R_f'=0$ and, by the effectivity, $K_X+D' \sim 0$. The same argument shows that $K_X+D\sim 0$. Thus $0\leq D-D'\sim 0$, hence $D=D'$.

By the assertion (1), $D$ contains a connected divisor $D''$ such that either $D''$ is a loop of smooth rational curves or $D''$ is irreducible and $p_a(D'') \ge 1$ (so that $h^0(X, K_X+D'') > 0$). By the above argument, $D = D''$. Now $K_X + D \sim 0$ and the equality (\ref{eq: section formula final}) above imply the second part of the assertion (2). This finishes the proof.
\end{proof}

\begin{lemma}\label{lem: no-isolated-minus-one}
Let $X$ be a smooth projective rational surface. Assume $X$ has only finitely many negative curves $C_j$, $1 \le j \le s$ (this occurs when $X$ has a non-isomorphism endomorphism). Let $D = \sum_{j=1}^s C_j$. Then we have:
\begin{enumerate}
\item[(1)]
If $D= 0$, then $X$ is isomorphic to $\mathbf{P}^1\times \mathbf{P}^1$ or $\mathbf{P}^2$. 
\item[(2)]
If $D$ is irreducible, then $X\cong \mathbf{F}_d$, the Hirzebruch surface with a minimal section of self-intersection $-d \le 0$.
\item[(3)]
$D$ is connected.
\item[(4)]
If $X$ admits a $q$-polarized endomorphism $f$, then for any two negative curves $C_i$, $C_j$ in $D$, we have $C_i\cdot C_j\le 1$.
\end{enumerate}
\end{lemma}

\begin{proof}
The assertions (1) and (2) are from the classification results for algebraic surfaces. We may assume that $D$ consists of at least two irreducible components.

For (3), we may run a $K_{X'}$-MMP and we will get a fibration structure $\pi_X: X'\to \mathbf{P}^1$. Every singular fiber of $\pi_X$ is connected, consists of negative curves and is hence contained in $D$. The (negative) minimal section $C_0$ (with $C_0^2 = -d$) has its proper transform $C_0'$ in $D$, which connects all singular fibers. Note that any $\pi_X$-horizontal negative curve $C$ in $D$ meets every fiber of $\pi_X$ and in particular all singular fibers. Since $D$ contains singular $\pi_X$-fibers and at least one $\pi_X$-horizontal component, $D$ is connected. This proves (3).

For (4), assume the contrary that $C_i\cdot C_j\ge2$ for some $i,j$. Then $C_i$ and $C_j$ form a loop. By Lemma~\ref{lem: finite neg curve}, the curves $C_i$ and $C_j$ are both $f^{-1}$-periodic. Thus $h^0(K_X+D)>0$ by Lemma~\ref{lem: loop criterion}(1). Hence $D=C_i+C_j$, by Lemma \ref{lem: loop criterion}(2). Since the one-point blow-up of $\mathbf{F}_d$ has at least $3$ negative curves, $X$ has at least $3$ negative curves, a contradiction.
\end{proof}

\subsection{Polarized endomorphism of essential dimension one}

\begin{definition}\label{defn: two maps are birational and birational base change}
 Let $f_i: X_i \dashrightarrow Y_i$ ($i = 1, 2$) be generically finite dominant rational maps. We say $f_1$ is \emph{birational to} $f_2$ if there exist birational maps $\sigma_X: X_1\dashrightarrow X_2$ and $\sigma_Y: Y_1 \dashrightarrow Y_2$ such that $\sigma_Y \circ f_1=f_2\circ \sigma_X$ as rational maps. We say $f_1$ is a \emph{birational base change} of $f_2$ if $f_1$ is birational to a pull-back of $f_2$ via some $\sigma_Y$.
\end{definition}

\begin{proposition}\label{prop: ed1 descend to P1}
Let $X$ be a birationally ruled smooth projective surface and let $f$ be a $q$-polarized endomorphism of $X$. Assume that the essential dimension $\ed(f) = 1$. Then we have:
\begin{enumerate}
\item[(1)]
$f$ is birational to the pull-back of a polarized endomorphism $g$ on $\mathbf{P}^1$. 
\item[(2)]
Suppose $f$ is Galois. Then we may choose $g$ such that $g$ is Galois and $\Gal(g)\cong \Gal (f)$.
\end{enumerate}
\end{proposition}
\begin{proof}

By assumption, we have a birational base change diagram:
 $$\xymatrix{
 X_1 \ar[r]^{f} \ar@{-->}[d]_{\sigma_1} & X_2 \ar@{-->}[d]^{\sigma_2} \\
 Y_1 \ar[r]^{g} & Y_2
 }$$
 where $X_1 = X$, $X_2 = X$, the horizontal maps are finite, $\sigma_i$ ($i=1,2$) are rational maps and $\dim (Y_1)=\dim (Y_2)=\mathrm{ed}(f)=1$. We may assume that
 \begin{enumerate}
 \item $Y_1, Y_2$ are smooth projective curves,
 \item the rational map $\sigma_2$ and hence $\sigma_1$ have irreducible general fibers, and
 \item $\deg(f)=\deg(g)$.
 \end{enumerate}
When $f$ is Galois, by \cite[Lemma~2.2]{BR97}, we may choose $Y_1,Y_2$ such that the function field extension $K(Y_1)/K(Y_2)$ is Galois, and $\Gal(g)\cong \Gal(f)$.

 If $X$ is rational, then both $Y_1$ and $Y_2$ are rational curves. This finishes the proof in this case. 
 
 From now on, we may assume that $X$ is irrational. By Lemma~\ref{lem: being ell ruled}, $X$ is a relatively minimal ruled surface over an elliptic curve $B$. Let $\pi: X\to B$ be the ruling. We have $q(X)=h^1(X,\mathcal{O}_X)=h^1(B,\pi_*\mathcal{O}_X)=h^1(B,\mathcal{O}_B)=1$. It follows that $q(Y_i)\leq q(X)=1$ for $i=1,2$ since the $Y_i$ are dominated by $X$. This implies that each $Y_i$ is either $\mathbf{P}^1$ or an elliptic curve. If $Y_1 \cong \mathbf{P}^1$, then $Y_2 \cong \mathbf{P}^1$ as it is dominated by $Y_1$, and we are done in this case. We are left with the following cases.

 \textbf{Case~1}: Both $Y_1$ and $Y_2$ are elliptic curves. In this case, both $\sigma_1$ and $\sigma_2$ are morphisms as every rational map from a smooth projective variety to an abelian variety extends to a morphism (cf. \cite[Chapter II, \S4, Corollary 2]{Mum74}). Let $h: X \to E:=\Alb(X)$ be the Albanese morphism. Since $q(X)=1$, the variety $E$ is an elliptic curve. By the universal property of the Albanese map (which also gives the ruling for $X$), $\sigma_1$ and $\sigma_2$ factor through $X\to E$. Since the general fibers of $\sigma_1$ and $\sigma_2$ are connected, $Y_i \cong E$ and $\sigma_i=h$ for $i=1,2$. By Lemma~\ref{lem: being ell ruled}, $g$ is a $q$-polarized endomorphism of $E$. This leads to a contradiction: $q^2=\deg(f)=\deg (g)=q$.

 \textbf{Case~2}: $Y_1$ is an elliptic curve and $Y_2$ is $\mathbf{P}^1$. As argued in \textbf{Case~1}, $\sigma_1$ is the ruling. Since a general fiber of $\sigma_2$ is dominated by a general fiber of $\sigma_1$ which is isomorphic to $\mathbf{P}^1$, the general fiber of $\sigma_2$ is isomorphic to $\mathbf{P}^1$. This implies that $X$ is rational, a contradiction.
\end{proof}

\section{Polarized endomorphisms of rational surfaces}\label{sect: rat surf}

In this section, we focus on polarized endomorphisms on rational surfaces. We first study the case of polarized endomorphisms of $\mathbf{P}^2$ (see \cite{LZ26} for an alternative proof).

\begin{proposition}\label{prop: ed for Galois for projective plane}
    Let $f: \mathbf{P}^2\to \mathbf{P}^2$ be a Galois cover of degree $>4$. Then $\ed(f)=2$.
\end{proposition}

\begin{proof}
    Let $G$ be the Galois group of $f$. Write $f^*\mathcal{O}(1) = \mathcal{O}(q)$ with $q > 2$.
    
    Suppose that $G$ fixes a point $x$ on $\mathbf{P}^2$. Then $G$ acts faithfully on the tangent space $T_x$. Since $\mathbf{P}^2/G=\mathbf{P}^2$ is still smooth, by the Chevalley-Shephard-Todd theorem and faithful representation lemma (cf.~Theorem \ref{lem: generated by pref} and Lemma \ref{lem: reductive faithful act}), $G\cong G|_{T_x}\subset \mathrm{GL}(2, {\bold k})$ is generated by pseudo-reflections. 
    
We continue to prove the proposition.
Assume to the contrary that $\ed(f)=1$. Then we have a birational base change diagram (see Definition~\ref{defn: two maps are birational and birational base change}): 
    $$\xymatrix{
    \mathbf{P}^2 \ar[r]^{f} \ar@{-->}[d] & \mathbf{P}^2 \ar@{-->}[d] \\
    \mathbf{P}^1 \ar[r]^{h} & \mathbf{P}^1
    }$$
    where the horizontal maps are finite. Moreover, $G$ acts faithfully on $\mathbf{P}^1$. By the classical classification of finite subgroups of $\mathrm{PGL}(2, {\bold k})$ (\cite[Introduction 3.1]{Bea10}), $G$ is isomorphic to $A_4, S_4, A_5, D_{2m}$ or $C_m$ with $|G|>4$. Since $|G| = \deg f = q^2$ is a square number, $G$ can only be $D_{2m}$ or $C_m$.

   Consider the case $G=C_m=\langle g\rangle$ ($m \ge 2$). Without loss of generality, we may assume that $g$ is represented by a diagonal matrix $\mathrm{diag}(\lambda_0,\lambda_1,\lambda_2)\in \mathrm{GL}(3,\mathbf{k})$. Then $g$ fixes three points $p_0=[1:0:0]$, $p_1=[0:1:0]$ and $p_2=[0:0:1]$ on $\mathbf{P}^2$. At each fixed point $p_i$, $G$ acts faithfully on the tangent space $T_{p_i}$ as a pseudo-reflection group. In particular, this implies that
    $$\langle\diag(\frac{\lambda_1}{\lambda_0},\frac{\lambda_2}{\lambda_0})\rangle,\ \ \ \langle\diag(\frac{\lambda_0}{\lambda_1},\frac{\lambda_2}{\lambda_1})\rangle,\ \ \ \langle\diag(\frac{\lambda_0}{\lambda_2},\frac{\lambda_1}{\lambda_2})\rangle$$ are all pseudo-reflections subgroups of $\mathrm{GL}(2,\mathbf{k})$. Let $a,b,c$ be the orders of $\frac{\lambda_0}{\lambda_1}, \frac{\lambda_1}{\lambda_2}, \frac{\lambda_2}{\lambda_0}$ respectively. The fact that these are pseudo-reflection groups implies $$\gcd(a,b)=\gcd(b,c)=\gcd(c,a)=1$$ and $ab=bc=ca=\deg(f)$, which leads to a contradiction.
    
    Consider the case $G=D_{2m}$, a dihedral group for some integer $m$, where $$D_{2m}=\langle s,r \mid s^2=r^m=1,\ srs=r^{-1}\rangle.$$ 
   Recall that $\mathrm{Fix}(r)$ dentoes the pointwise fixed locus of $r$. 
    Since $$s\mathrm{Fix}(r)=\mathrm{Fix}(srs^{-1})=\mathrm{Fix}(r^{-1})=\mathrm{Fix}(r),$$ $\mathrm{Fix}(r)$ is $s$-stable. In particular, there exists a point $x\in \mathrm{Fix}(r)$ such that the involution $s$ fixes $x$. This is because either $r$ is a pseudo-reflection and $\mathrm{Fix}(r)$ consists of a line and an isolated point or $\Fix(r)$ consists of three isolated points. Hence $x$ is fixed by $G$. Now $G \cong G|_{T_x}\subset \mathrm{GL}(2,\mathbf{k})$ is generated by pseudo-reflections whose degrees of the generators of the invariant ring $k[x,y]^G$ are at most $q=\sqrt{|G|}$, since $f$ is defined by degree $q$ polynomials. By Shephard-Todd's classification of pseudo-reflection groups \cite{ST54}, this happens only when $m=2$ and $G\cong (\mathbb{Z}/2\mathbb{Z})^2$ which implies $\deg f = 4$, a contradiction. This finishes the proof.
\end{proof}

\medskip

Let $\mathbf{F}_d$ be the Hirzebruch surface with the ruling $\pi:\mathbf{F}_d \to \Gamma \cong \mathbf{P}^1$. Let $C_0$ be the minimal section such that $C_0^2=-d$ and let $F$ be a fiber of $\pi$. In particular, when $d=0$, $\mathbf{F}_0\cong\mathbf{P}^1\times \mathbf{P}^1$.

\begin{theorem}\label{thm:ed of F_0}
Let $X \cong \mathbf{P}^1 \times \mathbf{P}^1$ and $f: X \to X$ be a $q$-polarized endomorphism. Suppose that $f$ preserves each of the two projections $p_j: X \to \mathbf{P}^1$, $f$ is Galois with $\Gal(f) = G$ and $|G| = \deg f > 4$. Then the essential dimension $\ed(f) = 2$.
\end{theorem}

\begin{proof}
 Assume to the contrary that $\operatorname{ed}(f) = 1$. We have the rational base change diagram:
 $$\xymatrix{
 \mathbf{P}^1\times\mathbf{P}^1 \ar[r]^{f} \ar@{-->}[d]_{\sigma_1} & \mathbf{P}^1\times\mathbf{P}^1 \ar@{-->}[d]^{\sigma_2} \\
 \mathbf{P}^1 \ar[r]^{g} & \mathbf{P}^1
 }$$
 such that $g$ is a Galois cover with the Galois group $G$. In particular, $G\leq \mathrm{PGL}(2, {\bf k})$. Since $f$ preserves the two projections from $\mathbf{P}^1\times\mathbf{P}^1$, $f$ splits as $f_i: \mathbf{P}^1\to \mathbf{P}^1$ for $i=1,2$. Moreover, let $G_j \le \PGL(2, {\bf k})$ ($j = 1, 2$) be the groups induced by $G$ on the  two bases of the two projections of $X$
 so that $G \le G_1 \times G_2$. By \cite[Theorem~1.3]{MZ18}, $f_i$ is $q$-polarized for each $i$. Since $f_i$ is just the quotient map ${\bf P}^1 \to {\bf P}^1/G_i$, we have $|G_i| = q$. Comparing orders, we have $G = G_1 \times G_2$.
By the classical classification of finite subgroups of $\mathrm{PGL}(2, {\bf k})$ (\cite[Introduction]{Bea10}) and the fact that $G$ is the direct product of two order-$q$ subgroups of $\mathrm{PGL}(2, {\bf k})$, we have that $G\cong \mathbb{Z}/2\mathbb{Z}\times \mathbb{Z}/2\mathbb{Z}$. This contradicts our assumption that $|G|>4$.
\end{proof}

\begin{remark}\label{preserve projection}
The assumption that $f$ preserves each of the two projections of $\mathbf{P}^1 \times \mathbf{P}^1$ is necessary. Indeed, as in \cite[Example 5.13]{LOZ25}, for any $q \ge 2$, we define $f: (x,y)\mapsto (y,x^{q^2})$. Then $f$ is $q$-polarized and cyclic Galois with $\deg f = q^2$. Thus $\ed(f) = 1$ by Lemma \ref{lem: facts of essential dimension}(3).
\end{remark}

\medskip

In the rest of this section, we will deal with general rational surfaces.

\begin{lemma}\label{lem: pullback is self}
Let $X$ be a smooth projective surface, and let $f: X \to X$ be a $q$-polarized Galois endomorphism with cyclic Galois group $\mathrm{Gal}(f) = \langle g \rangle$. Let $C \subset X$ be an irreducible smooth curve such that $f^{-1}(C) = C$. Then $f^* C = q C$, and $C$ is an elliptic curve.
\end{lemma}

\begin{proof}
Since $f^{-1}(C) = C$, we have that $f^* C = rC$ for some $r>0$.
As $f$ is $q$-polarized, it follows that $r = q$, hence $f^* C = qC$. Since the restriction $f|_C$ is a $q$-polarized endomorphism of $C$, the genus of $C$ is either $0$ or $1$ by \cite[Theorem~1.3]{NZ}.

Note that $g(C)=C$. Assume that $C \cong \mathbf{P}^1$. Then $g|_C$ has a fixed point $P \in C$. 
After choosing local coordinates $x,y$ at $P$, we may write the action of $g$ as $g(x,y)=(\lambda x,\mu y)$, where $\lambda,\mu$ are powers of a primitive $q^2$-th root of unity $\zeta$ and $C$ is defined by $\{y=0\}$. Since $f^*C=qC$, $\lambda$ is a primitive $q$-th root of unity. Since the order of $g$ is $q^2$, $\mu$ is a primitive $q^2$-th root of unity. After taking a $G$-stable affine neighborhood of $P$, Lemma \ref{lem: generated by pref} implies that $G$ is generated by pseudo-reflections. Since $g^r$ is a pseudo-reflection if and only if $q|r$, $\langle g\rangle$ is not generated by pseudo-reflections. This yields a contradiction. Therefore, $C$ is an elliptic curve.
\end{proof}

\begin{lemma}\label{lem:relative-minimal}
Let $X$ be a smooth projective rational surface and let
$f\colon X\to X$ be a $q$-polarized cyclic Galois endomorphism with $\deg f \ge 4$. Then $X\cong \mathbf{P}^2$ or $X\cong \mathbf{F}_d$ with $d\geq 0$.
\end{lemma}

\begin{proof}
Suppose $X$ is not isomorphic to ${\bf P}^2$ or ${\bf F}_d$. Then $X$ is the blow-up of ${\bold F}_d$. Let $D$ be the reduced divisor consisting of all negative curves on $X$. Since there exists a $(-1)$-curve, Lemma~\ref{lem: loop criterion} implies that every irreducible component of $D$ is isomorphic to $\mathbf{P}^1$.
By Lemma \ref{lem: no-isolated-minus-one}, the conditions of Lemma~\ref{lem: 2-Go} are satisfied.
Since $\deg(f)=q^2\ge 4$ and by Lemma~\ref{lem: 2-Go},
$D$ contains an $f^{-1}$-stable curve $D_1$. This contradicts Lemma~\ref{lem: pullback is self} as $D_1$ is a rational curve. 
\end{proof}

\begin{theorem}\label{thm:ed-two}
Let $X$ be a smooth projective rational surface, and let
$f\colon X\to X$ be a $q$-polarized endomorphism .
Assume that $X$ is not isomorphic to $\mathbf{P}^2$ or $\mathbf{P}^1\times \mathbf{P}^1$, $f$ is Galois, and $\deg(f)>4$. Then $f$ is incompressible.
\end{theorem}

\begin{proof}
Note that $\ed(f)\geq 1$. Assume that $\ed(f)=1$. We proceed in the following steps to deduce a contradiction. Let $G$ be the Galois group of $f$.

 By Proposition~\ref{prop: ed1 descend to P1}, after a rational base change, $f$ descends to a Galois endomorphism of $\mathbf{P}^1$ whose Galois group is isomorphic to $G$. By Lemma~\ref{lem: being cyc or dih}, the group $G$ is a cyclic group or a dihedral group.
 \begin{claim}\label{claim: cyclic group ruled surface}
 $G$ is a cyclic group.
 \end{claim}
\begin{proof}[Proof of Claim~\ref{claim: cyclic group ruled surface}]
 Assume that $G$ is a dihedral group. We will deduce a contradiction.

 We first treat the case when $X\cong \mathbf{F}_d$ for some $d\geq 1$.
 Let $C_0$ be the unique negative section of $\mathbf{F}_d$, and let $\pi\colon X\to \mathbf{P}^1$ be the ruling. Then $f^{-1}(C_0)=C_0$, $C_0$ is $G$-stable and $G$ preserves the ruling $\pi$. We may write $f^*(C_0)=q'C_0$. Since $f$ is $q$-polarized,
 we have $q'=q$. Then, by the projection formula, $\deg(f|_{C_0})=q$. Set $H:=\Ker(G\to G|_{C_0}).$ Then $|H|=q$. We use the notation $G = D_{2n} = \langle \sigma, \tau \rangle$ in Lemma \ref{lem: being cyc or dih} where $2n = q^2$. Since $C_0$ is $\sigma$-stable, $\sigma$ fixes some point $Q \in C_0$.
 Note that $G_Q$ stabilizes $C_0$ and the fiber $F_Q$ through the point $Q$, so $G_Q$ is abelian by Lemma~\ref{lem: stabilizer-abelian-at-node}. Thus $G = D_{2n}$ (with $n > 2$) contains an abelian subgroup $G_Q$ ($\ge \langle \sigma, H \rangle$) of order $\ge q > 2$.
 By the classification of abelian subgroups of $G$ in Lemma \ref{lem: being cyc or dih}, and since $q > 2$ and $\sigma \not\in \langle \tau \rangle$, we have 
 $H = \langle \sigma, H \rangle = \langle \sigma, \tau^{n/2} \rangle \cong \mathbb{Z}/2\mathbb{Z} \times \mathbb{Z}/2\mathbb{Z}$, and hence $q = 4, n = 8$. Since $\tau$ acts on $C_0$, it has a fixed point $Q' \in C_0$. Moreover, since $\sigma \in H$, the element $\sigma$ fixes $C_0$ pointwise. In particular, $\sigma(Q')=Q'$. Hence $G = D_{2n} = \langle \sigma,\tau \rangle \subseteq G_{Q'} \subseteq G$ which stabilizes $C_0$ and the fiber $F_{Q'}$ through $Q'$.  Lemma \ref{lem: stabilizer-abelian-at-node} implies that $G = G_{Q'}$ is abelian. This is absurd.

 Now we treat the case where $X$ is not $\mathbf{P}^2$ or $\mathbf{F}_d$ for any $d\geq 0$. So $X$ is the blow-up of some ${\bf F}_d$. By Lemma~\ref{lem: no-isolated-minus-one}, the reduced divisor $D$ consisting of all negative curves is connected and reducible. Hence, there exist two intersecting negative curves $D_1$, $D_2$ such that $D_1\cap D_2= \{Q\}$ for some point $Q\in X$ (cf.~Lemma~\ref{lem: no-isolated-minus-one}). Then $Q$ is fixed by $G$ since each $D_j$ is $G$-stable by Lemma~\ref{lem: finite neg curve}. Hence $G$ is abelian by Lemma~\ref{lem: stabilizer-abelian-at-node}. This contradicts the fact that a dihedral group of order $>4$ is not abelian. This finishes the proof of the claim.
\end{proof}

By Lemma~\ref{lem:relative-minimal}, $X\cong \mathbf{F}_d$ for some $d\geq 1$. Let $C$ be the unique negative section of $\mathbf{F}_d$. Since $C$ is the only negative curve on $X$, we have $f^{-1}(C)=C$. By Lemma~\ref{lem: pullback is self}, $C$ is an elliptic curve, a contradiction. This completes the proof.
\end{proof}

Now, we are ready to prove our main result for rational surfaces.
\begin{proof}[Proof of Theorem~\ref{Thm: main}]
 In the case where $X\ncong \mathbf{P}^1\times\mathbf{P}^1$, this follows from Proposition~\ref{prop: ed for Galois for projective plane} and Theorem \ref{thm:ed-two}. In the case where $X\cong \mathbf{P}^1\times\mathbf{P}^1$, it follows from Theorem~\ref{thm:ed of F_0}.
\end{proof}

\section{Polarized endomorphisms of irrational ruled surfaces}\label{sect: irrat surf}

In this section, we deal with polarized endomorphisms on smooth projective surfaces $X$ that are birationally ruled and irrational. By Lemma~\ref{lem: being ell ruled}, $X$ is a relatively minimal ruled surface over an elliptic curve $C$. So
$X=\mathrm{Proj}_C(\mathcal{E})$ for some rank two vector bundle
$\mathcal{E}$.

\begin{lemma}

\label{lem:p1-fixed-points-finite-order}
Let $q>2$ be an integer and let $K$ be a field of characteristic $0$ containing all $q$-th roots of unity. Let $h\in \operatorname{PGL}_2(K)$ be an element of order $q$. Then $h$ has two distinct $K$-rational fixed points on $\mathbf P^1_K$.

\end{lemma}

\begin{proof}
Choose a lift $h'\in \operatorname{GL}_2(K)$ of $h$. Since $h^q=1$ in $\operatorname{PGL}_2(K)$, there exists some $a\in K^*$ such that $h'^q=aI$. Hence, the minimal polynomial of $h'$ divides $T^q-a$, which has no repeated roots over $\overline K$. Thus, $h'$ is diagonalizable over $\overline K$.

Let $\ell_1,\ell_2\subset \overline K^2$ be the two eigenspaces of $h'$, with eigenvalues $\lambda_1,\lambda_2\in \overline K^*$. Since $h$ has order $q$, the ratio $\zeta:=\lambda_1/\lambda_2$ is a primitive $q$-th root of unity. By assumption, $\zeta\in K$.

Let $\Gamma_K=\operatorname{Gal}(\overline K/K)$. Since $h'$ is defined over $K$, the set of two eigenspaces $\{K\cdot\ell_1,K\cdot\ell_2\}$ is stable under $\Gamma_K$. Suppose that some $\sigma\in \Gamma_K$ exchanges $K\cdot\ell_1$ and $K\cdot\ell_2$. Then $\sigma(\lambda_1)=\lambda_2$ and $\sigma(\lambda_2)=\lambda_1$, so
$$
\sigma(\zeta)
=\sigma(\lambda_1/\lambda_2)
=\lambda_2/\lambda_1
=\zeta^{-1}.
$$
Since $\zeta\in K$, $\sigma(\zeta)=\zeta$. Thus $\zeta=\zeta^{-1}$, and it follows that $\zeta^2=1$, contradicting that $\zeta$ is a primitive $q$-th root of unity with $q>2$.

Therefore, every element of $\Gamma_K$ fixes both $K\cdot\ell_1$ and $K\cdot\ell_2$. By Galois descent for one-dimensional subspaces of $\overline K^2$, each $K\cdot\ell_i$ is defined over $K$. Hence, the two points of $\mathbf P^1_{\overline K}$ corresponding to $K\cdot \ell_1$ and $K\cdot \ell_2$ are distinct $K$-rational fixed points of $h$.
\end{proof}

\begin{remark}\label{rem:p1-order-two-obstruction}
The hypothesis that $q>2$ is necessary. For any $a\in K^*\setminus K^{*2}$, the matrix
$$
\begin{pmatrix}
0 & a\\
1 & 0
\end{pmatrix}
$$
has order $2$ in $\operatorname{PGL}_2(K)$, but its fixed points are defined over $K(\sqrt a)$ and not over $K$.
\end{remark}

\begin{lemma}\label{lem: fixed locus disjoint}
    Let $\pi:X\to C$ be a ruled surface over an elliptic curve $C$ such that there exists an element $h\in \mathrm{Aut}(X)$ of order $q>2$ with $\pi\circ h=\pi$. Then the fixed locus of $h$ is a disjoint union of two sections of $\pi$. In particular, $\pi_*\mathcal{O}_X(1)$ splits.
\end{lemma}
\begin{proof}
    Let $\eta=\mathrm{Spec}K(C)$ be the generic point of $C$. Restricting $h$ to the generic fiber $X_\eta$ gives a $K(C)$-automorphism $h_\eta:X_\eta\to X_\eta$. Since $X\to C$ is a ruled surface, $X_\eta\cong \mathbf{P}_{K(C)}^1$. Since ord$(h)=q>2$, ord$(h_\eta)>2$. Therefore, $h_\eta$ has two distinct fixed $K(C)$-points by Lemma~\ref{lem:p1-fixed-points-finite-order}, which give two rational sections $C\dashrightarrow X $. Since $C$ is a smooth projective curve and $X\to C$ is proper, each rational section extends uniquely to a section. Thus, we obtain two distinct sections $C_1,C_2\subset X$ that are $h$-stable, hence fixed by $h$ pointwise.
    
    Note that $C_1,C_2\subset \Fix(h)$. Each general fiber of $\pi$ has exactly two $h$-fixed points. Thus $\Fix(h)$ is a disjoint union of $C_1$ and $C_2$ by Lemma \ref{lem: reductive faithful act} and the fact that $\Fix(h)$ is smooth. 
\end{proof}

\begin{corollary}\label{cor: split vector bundle}
    Let $X=\mathrm{Proj}_C(\mathcal{E})$ be a ruled surface over an elliptic curve $C$, with $\mathcal{E} $ a rank two vector bundle on $C$. Let $f: X\to X$ be a $q$-polarized Galois endomorphism such that $q>2$ and the Galois group $G$ is either cyclic or dihedral. Then $\mathcal{E}$ splits.
\end{corollary}

\begin{proof}
    Let $\pi: X\to C$ be the Albanese morphism, which is just the ruling of $X$. By Lemma~\ref{lem: being ell ruled}, $f$ descends to a $q$-polarized endomorphism $f|_C$ on $C$. 
We have an induced action $G|_C$ on $C$ which is abelian of order equal to $\deg f|_C = q > 2$. Let $H :=\Ker(G \to G|_C)$. Then $|H| = q$.

If $G$ is a cyclic group, then $H$ is cyclic. If $G$ is a dihedral group, by Lemma~\ref{lem: being cyc or dih}, we have $q=4$, $G=D_{16}$, and $H$ is a cyclic group of order $4 = q$. 

Thus, by Lemma~\ref{lem: fixed locus disjoint}, $X$ has two disjoint sections. This implies that $\mathcal{E}$ splits.
\end{proof}

To study polarized Galois covers of elliptic ruled surfaces in detail, we need the following.

\begin{lemma}\label{lem: no dihedral Galois cover for elliptic ruled}
    Let $\pi:X\to C$ be a ruled surface over an elliptic curve. Then there is no $q$-polarized Galois endomorphism $f:X\to X$ with $q>2$ whose Galois group is a dihedral group.
\end{lemma}

\begin{proof}
The proof is by contradiction. Assume that such an endomorphism $f:X\to X$ exists, and denote its dihedral Galois group by $G$. As in Corollary \ref{cor: split vector bundle}, $f$ descends to a $q$-polarized endomorphism $f|_C$ on $C$, we have an induced action $G|_C$ on $C$ which is abelian of order equal to $\deg f|_C = q = 4$, and $H:=\Ker(G\to G|_C)=\langle h\rangle$ is a cyclic group of order $q = 4$. In addition,
$$G=D_{16}:=\langle \tau,\sigma\mid
\tau^8=\sigma^2=1,~~\sigma\tau\sigma=\tau^{-1}\rangle,$$ $G|_C\cong G/H = \langle \bar{\tau}, \bar{\sigma}\rangle \cong \mathbb{Z}/2\mathbb{Z}\times \mathbb{Z}/2\mathbb{Z}$ and $H=\langle h = \tau^2\rangle$.
Here for $g \in G$, we let $\bar{g}$ be its image in $G/H$.

By Lemma~\ref{lem: fixed locus disjoint}, we may write $\Fix(h)=C_1\cup C_2$, with $C_1,C_2$ two disjoint sections. Since $h$ fixes $C_i$ pointwise, it induces an automorphism of normal bundles $\mathcal{N}_{C_i/X} $ on $C_i$. Because $C_i$ is projective and connected, every automorphism of the line bundle is multiplication by a scalar. Hence, there are $\lambda_i\in \mathbf{k}^*$ such that $h|_{\mathcal{N}_{C_i/X}}=\lambda_i \mathrm{id}$. Take a closed point $P\in C_i$, and let $F_P$ be the fiber of $\pi$ through $P$. Since $T_PX=T_PC_i\oplus T_PF_P$ and $h$ acts trivially on $C_i$, we have $h|_{T_PX}=h|_{T_P C_i}\oplus h|_{T_P F_P}=\mathrm{id}\oplus \lambda_i \mathrm{id}$. By Lemma~\ref{lem: reductive faithful act}, the tangent representation is faithful. Hence $\lambda_i$ is a primitive $4$-th root of unity. 

Since $\sigma h\sigma=h^{-1}$, the involution $\sigma$ either preserves both $C_j$ of $\Fix(h)$ or swaps them. If $\sigma(C_j) = C_j$ ($j = 1, 2$), the equality $\sigma h\sigma=h^{-1}$ implies $\lambda_i=\lambda_i^{-1}$ which is a contradiction. Hence, $C_1,C_2$ are swapped by $\sigma$, and their self-intersection numbers are the same. It follows that $X=\Proj_C(\mathcal{O}_C\oplus\mathcal{L})$, with $\mathcal{L}$ a degree $0$ line bundle. Moreover, the normal bundles of the two minimal sections are isomorphic. Thus $\mathcal{L}\cong\mathcal{L}^{-1}$. Therefore, $\mathcal{L}$ is a $2$-torsion line bundle.

Let $\rho: X\to X/H =: Y$ which is smooth because $\Fix (h)$ is a smooth divisor and by Theorem \ref{lem: generated by pref}.  Since $H$ acts trivially on $C$, the quotient $Y=X/H$ admits an induced ruling $Y\to C$. Let 
$$D_i:=\rho(C_i)\subset Y.$$ Since $H$ fixes $C_i$ pointwise, $\rho|_{C_i}$ is an isomorphism and $\rho^*D_i=4C_i$. The projection formula implies that the $D_j$ are sections of $Y \to C$ and disjoint. 
Moreover, restricting to $C_i$, we get the pullback of normal bundles $$\mathcal{N}_{D_i/Y}\cong(\rho|_{C_i})^*\mathcal{N}_{D_i/Y}\cong\mathcal{O}_X(\rho^*D_i)|_{C_i}\cong \mathcal{N}_{C_i/X}^{\otimes4}\cong \mathcal{O}_{C_i}$$ where the last isomorphism holds because $\mathcal{L}$ is a $2$-torsion. Thus we have an identification $Y = C\times\mathbf{P^1}$. 

Let $Z:=Y/(G/H)\cong X$. There is a naturally induced ruling $Z\to C/(G|_C)$ which can be identified as $X\to C$. The ruling $Y \to C$ is obtained from the ruling on $Z$ via the \'etale base change $C \to C/(G|C)$. Our $Z\cong X$ has a section $D$ with $D^2=0$. Pulling back via $Y \to Z$
gives a section $\widetilde D$ of $C\times \mathbf{P}^1\to C$ such that $\widetilde D^2=4 D^2=0$. Hence $\widetilde D=C\times\{p\}$ for some $p\in \mathbf{P}^1$.

Note that each $\bar{g} \in G/H$ (acting on $Y = C \times {\bold P}^1$) induces $\bar{g}|_C$ and $\bar{g}|_{{\bold P}^1}$. Since $C\times\{p\} = \widetilde{D}$ is pulled back via $Y \to Y/(G/H) = Z$, it is $G/H$-stable,
so $p$ is fixed by $(G/H)|_{{\bold P}^1}$. 

Since  $\mathcal{O}_X(C_1-C_2)\cong \pi^*\mathcal{L}^{\pm 1}$ and $\mathcal{L}$ is $2$-torsion, there exists some $u\in k(X)^*$ such that $$\operatorname{div}(u)=2(C_1-C_2).$$ We first prove the claim that
\begin{equation}\label{eq: pull back minus}
    \hskip 1pc h^*u=-u.
\end{equation} Since $h$ fixes $C_1$ and $C_2$ pointwise
and $\operatorname{div}(u)=2(C_1-C_2)$, we have
$\operatorname{div}(h^*u)=\operatorname{div}(u)$. Hence
$h^*u/u\in H^0(X,\mathcal O_X^*)={\bold k}^*$, so, for some $c\in {\bold k}^*$, we have $$h^*u=cu.$$ 
Choose a general point $y\in C$ and put $F:=\pi^{-1}(y)\simeq \mathbf P^1$.
Then $h|_F$ has order $4$. Let $p_i:=C_i\cap F$ for
$i=1,2$. Since $h$ fixes $C_1$ and $C_2$ pointwise, $h|_F$ fixes both
$p_1$ and $p_2$.
Choose $t\in k(F)^*$ such that $\operatorname{div}_F(t)=p_1-p_2$. Then
$\operatorname{div}_F((h|_F)^*t)=p_1-p_2$, hence $(h|_F)^*t=\alpha t$ for some
$\alpha\in \mathbf{k}^*$. Since $h|_F$ has order $4$, $\alpha$ is a primitive fourth
root of unity, so $\alpha^2=-1$.

Restricting $\operatorname{div}(u)=2(C_1-C_2)$ to $F$, we get
$
\operatorname{div}_F(u|_F)=2(p_1-p_2)=\operatorname{div}_F(t^2).
$
Thus $u|_F=bt^2$ for some $b\in {\bold k}^*$. Therefore
$$
(h|_F)^*(u|_F)=(h|_F)^*(bt^2)=b(\alpha t)^2=\alpha^2bt^2=-u|_F.
$$
On the other hand, restricting $h^*u=cu$ to $F$ gives
$h_F^*(u|_F)=c\,u|_F$. Hence $c=-1$, and therefore we have proved the claim $h^*u=-u$.

We continue the proof of the lemma. Note that $\tau$ preserves $\Fix(\tau^2 = h) = C_1 \cup C_2$. Suppose that $\tau$ exchanges $C_1$ and $C_2$. Then
$\operatorname{div}(\tau^*u)=-\operatorname{div}(u)$, so $\tau^*u=a/u$ for
some $a\in \mathbf{k}^*$. Hence $(\tau^2)^*u = u$.
This and $\tau^2=h$ contradict the claim (\ref{eq: pull back minus}) above. Therefore $\tau(C_j) = C_j$ ($j = 1, 2$).

Consequently, $\operatorname{div}(\tau^*u)=\operatorname{div}(u)$, and hence
$\tau^*u=au$ for some $a\in \mathbf{k}^*$. Using again $\tau^2=h$, we get
$a^2u=(\tau^2)^*u=h^*u=-u$. Thus $a^2=-1$.
Since $\rho^*D_1=4C_1$ and $\rho^*D_2=4C_2$, we have
$$
\operatorname{div}(u^2)=4(C_1-C_2)=\rho^*(D_1-D_2).
$$
Also, by the claim \eqref{eq: pull back minus}, $u^2\in K(X)^H=K(Y)$. Hence there is an
$s\in K(Y)^*$ such that $\rho^*s=u^2$. This and the above equality give
$\operatorname{div}_Y(s)=D_1-D_2$. Since $D_i^2=0$, we may write $D_i=C\times\{q_i\}$ ( $i=1,2$). 
So we may regard $s$ as an element of
$K(\mathbf P^1)^* \subset K(C \times {\bold P}^1)^* = K(Y)^*$. Now
$$
\rho^*(\bar\tau^*s)=\tau^*(\rho^*s)=\tau^*(u^2)=a^2u^2=-u^2=\rho^*(-s).
$$
Since $\rho^*$ is injective, $\bar\tau^*s=-s$. Thus $\bar\tau$ acts
nontrivially on the $\mathbf P^1$-factor of $Y = C \times {\bf P}^1$.

Since $\sigma$ exchanges $C_1$ and $C_2$, its image $\bar\sigma \in G/H$ exchanges the $D_j = C \times \{q_j\}$, because
$$
\bar\sigma(D_i)=\bar\sigma(\rho(C_i))=\rho(\sigma(C_i))=D_{3-i}
$$
for $i=1,2$. Hence
\begin{equation}\label{eq: restriction of morphism}
    \hskip 1pc \bar{\sigma}|_{{\bold P}^1} (q_i) = q_{3-i}
\end{equation}
for $i = 1, 2$. On the other hand, since $\bar\tau$ preserves both $D_i$, $\bar{\tau}|_{{\bold P}^1} $ fixes both $q_i$. Since $\bar{\tau}$ acts non-trivially on the $\mathbf P^1$-factor we have $\Fix(\bar{\tau}|_{{\bold P}^1}) = \{q_1, q_2\} \supset \Fix(\bar{\tau}|_{{\bold P}^1}) \cap \Fix(\bar{\sigma}|_{{\bold P}^1})  \supset \{p\}$; here we recall that $\widetilde{D} = C \times \{p\}$ which is $G/H$-stable.
This contradicts the equality \eqref{eq: restriction of morphism} above. \end{proof}

\begin{lemma}\label{lem: a bound for q for splitting bundle}
 Let $X = \mathrm{Proj}_C(\mathcal{O}_C \oplus \mathcal{L})$ be a ruled surface over an elliptic curve $C$, where $\mathcal{L}$ is a line bundle on $C$. Let $f: X\to X$ be a $q$-polarized Galois endomorphism such that the Galois group $G$ is cyclic and $q > 2$. Then the following statements hold.
 \begin{enumerate}
 \item Suppose $\deg \mathcal{L} \neq 0$. Then $q \mid \deg \mathcal{L}$.
 \item Suppose $\deg \mathcal{L} = 0$. Then $\mathcal{L}$ is a torsion line bundle and $q \mid \mathrm{ord}(\mathcal{L})$.
 \end{enumerate}
\end{lemma}

\begin{proof}
By Lemma~\ref{lem: being ell ruled}, $f$ descends to a $q$-polarized morphism $f|_C: C\to C$ via the Albanese map $\pi: X\to C$.
Both $G|_C$ and $H:  =\Ker(G \to G|_C)$ are isomorphic to $\mathbb{Z}/q\mathbb{Z}$. We may assume that $f|_C$ is an isogeny, otherwise, we just compose it with a translation, and the same proof works. We may write $G|_C=\langle t_c\rangle$, where $t_c$ denotes the translation morphism for some $c\in C$.

Note that $\mathrm{Proj}_C(\mathcal{O}_C\oplus t_c^*\mathcal{L})\cong\mathrm{Proj}_C(\mathcal{O}_C\oplus \mathcal{L})$ over $C$. Hence $\mathcal{O}_C\oplus t_c^*\mathcal{L}\cong(\mathcal{O}_C\oplus \mathcal{L})\otimes \mathcal{M}$ for some invertible sheaf $\mathcal{M}$. This and Krull-Schmidt theorem \cite[Part 1, Generalities]{Ati57} imply that $t_c^*\mathcal{L}\cong \mathcal{L}$ or $t_c^*\mathcal{L}\cong \mathcal{L}^{-1}$.

In the case when $d: = \deg \mathcal{L}\neq 0$, since the degrees of both $t_c^*\mathcal{L}$ and $L$ have the same sign, only $t_c^*\mathcal{L}\cong \mathcal{L}$ can occur. Thus $$G|_C\subseteq K(\mathcal{L}):=\Ker(\lambda_{\mathcal{L}})=\{c\in C:t_c^*\mathcal{L}\otimes \mathcal{L}^{-1}\cong \mathcal{O}_C\},$$
where $\lambda_{\mathcal{L}}$ is the polarization morphism. 
By a direct computation, under the standard identification $C\simeq \operatorname{Pic}^0(C)$, the morphism $\lambda_{\mathcal L}$ agrees with the multiplication by $d$ map.  Hence
$$
K(\mathcal L)=\Ker(\lambda_{\mathcal L})=C[d],
$$
where $C[d]:=\Ker([d]:C\to C)$ is the group of $d$-torsion points of $C$. So $G|_C$ is a subgroup of $C[d]$ . Since $G|_C$ has an element of exact order $q$, we get $q\mid d$. This proves the assertion (1).

(2) In the case when $\deg \mathcal{L}=0$ (here we remark that $\mathcal{L}$ could be the trivial line bundle), by Theorem~\ref{thm: classification of dynamical surface}(4), $\mathcal{L}$ is a torsion line bundle. Let $$\ell:=\mathrm{ord}(\mathcal{L}).$$ We may write $G=\langle t\rangle$ such that $t^{q^2}=\mathrm{id}$. Then $H=\langle h:=t^q\rangle$. By Lemma~\ref{lem: fixed locus disjoint}, $\Fix(h)$ is exactly the union of two disjoint sections. We denote them by $S_0$ and $S_\infty$. Since they are disjoint, a
numerical-class computation gives $S_\infty^2=-S_0^2$.

On the other hand, since $X=\mathbf P_C(\mathcal O_C\oplus \mathcal L)$ and
$\deg \mathcal L=0$, every section has nonnegative self-intersection.
Thus $S_0^2=S_\infty^2=0$. Hence $S_0$ and $S_{\infty}$ are the only minimal sections when $\mathcal{L}$ is non-trivial, and are equal to $C \times \{0\}$ and $C \times \{\infty\}$ in suitable coordinates when $\mathcal{L} = \mathcal{O}_C$ and hence $X = C \times \bold{P}^1$.
So $$\mathcal{O}_X(S_\infty-S_0)\cong\pi^*\mathcal{L}^{\pm 1}.$$ In particular, $\mathcal{O}_X(S_\infty-S_0)$ has exact order $\ell$.
Hence there exists some $u\in K(X)^*$ such that 
\begin{equation}\label{eq: an equation}
    \mathrm{div}(u)=\ell(S_\infty-S_0).
\end{equation}
 On each fiber $F$ of $\pi$, $h$ acts as $z\mapsto \xi_q z$ in suitable coordinates, where $\xi_q$ is a primitive $q$-th root of unity.
Moreover, the intersections of $F$ with the two minimal sections correspond to the points defined by $z=0$ and $z=\infty$. Next we prove:

\begin{claim}\label{claim: pull back invariant}
    $h^*u=(t^q)^*u=u.$
\end{claim}
\begin{proof}[Proof of Claim~\ref{claim: pull back invariant}]
Since $t$ preserves $\Fix(h)$, it either preserves both $S_0$ and $S_\infty$, or swaps them. 

If $t$ swaps $S_0$ and $S_\infty$, then $q$ is even.
In this case, $$\mathrm{div}(t^*u) = t^*\mathrm{div}(u) = -\ell(S_\infty-S_0).$$
Hence $t^*u=a/u$ for some $a\in \mathbf{k}^*$. It follows that $h^*u=(t^q)^*u=\bigl((t^2)^{q/2}\bigr)^*u=u.$ 

If $t$ preserves both $S_0$ and $S_\infty$, then $\mathrm{div}(t^*u)=\mathrm{div}(u)$, 
and hence 
$$t^*u=bu$$ for some $b\in \mathbf{k}^*$. Set $T_0=f(S_0)$ and $T_\infty=f(S_\infty)$. Then $f^*T_0=qS_0$ and $f^*T_\infty=qS_\infty.$ This and the projection formula imply that $T_0$ and $T_{\infty}$ are disjoint minimal sections, so $\mathcal{O}_X(T_\infty-T_0)$ is also an $\ell$-torsion.
Hence, there exists some $v\in K(X)^*$ such that $\mathrm{div}(v)=\ell(T_\infty-T_0).$
Pulling back this equality by $f$, we get
$$\mathrm{div}(f^*v) = \ell(f^*T_\infty-f^*T_0) = \ell(qS_\infty-qS_0) = \mathrm{div}(u^q).$$ Therefore $f^*v=\alpha u^q$ for some $\alpha\in \mathbf{k}^*$. Since $f\circ t=f$, we have $t^*(f^*v)=f^*v.$ Hence
$$\alpha u^q = f^*(v) = t^*(f^*(v)) = t^*(\alpha u^q) = \alpha (t^*u)^q = \alpha b^q u^q.$$
It follows that $b^q=1$ and $h^*u=(t^q)^*u=b^q u=u.$
This proves the claim.
\end{proof}

We continue the proof of the lemma. Restricting the equation $h^*u=u$ to a general fiber $F$ of $\pi$, we have
$$u|_F = h^*(u|_F) = \beta(\xi_q z)^{-\ell} = \xi_q^{-\ell}\beta z^{-\ell} = \xi_q^{-\ell}u|_F;$$ here we choose coordinate $z$ on $F$ so that $S_0\cap F=\{z=0\}$, $S_\infty\cap F=\{z=\infty\},$ $h|_F:z\mapsto \xi_q z$ and then $u|_F=\beta z^{-\ell}$ for some $\beta\in \mathbf{k}^*$ by the equality \eqref{eq: an equation}. It follows that $\xi_q^{\ell}=1.$
Since $\xi_q$ is a primitive $q$-th root of unity, we get $q\mid \ell$. This proves (2), and hence the lemma.
\end{proof}

Now we are ready to prove Theorem~\ref{thm: elliptic ruled} and Remark \ref{rem: low bd of deg for ed}.

\begin{proof}[Proof of Theorem~\ref{thm: elliptic ruled} and Remark \ref{rem: low bd of deg for ed}]
By Lemma~\ref{lem: being ell ruled}, $X$ is a ruled surface over an elliptic curve $C$, and we may write $X=\mathrm{Proj}_C(\mathcal{E})$ for some rank two vector bundle $\mathcal{E}$ on $C$.

Assume for the sake of contradiction that $\mathrm{ed}(f)=1$. By Proposition~\ref{prop: ed1 descend to P1}, $G$ is a finite subgroup of $\mathrm{PGL}(2, {\bf k})$ of order $\deg(f)=q^2$ for some integer $q>2$. By Lemma~\ref{lem: being cyc or dih}, $G$ is either the dihedral group $D_{q^2}$ or the cyclic group $C_{q^2}$. By Corollary~\ref{cor: split vector bundle}, $\mathcal{E}$ splits. This proves (1).

We may write $\mathcal{E}=\mathcal{O}_C\oplus \mathcal{L}$. By Lemma~\ref{lem: no dihedral Galois cover for elliptic ruled}, we may assume that $G$ is $C_{q^2}$. By Lemma~\ref{lem: a bound for q for splitting bundle}, there exists an integer $N$ such that for any $q>N$, $f$ cannot be a cyclic Galois cover. To be more specific, if $\deg(\mathcal{L})\neq 0$, we can take $N=\max\{2, {|\deg(\mathcal{L})|}\}$, and if $\deg(\mathcal{L})= 0$, then $\mathcal{L}$ is a torsion line bundle by Theorem~\ref{thm: classification of dynamical surface}(4), and we may take $N=\max\{2,\mathrm{ord}(\mathcal{L})\}$. This leads to a contradiction if we assume that $q>N$, and hence proves (2).
\end{proof}

\begin{remark}\label{rem: optimal bounds for elliptic ruled surface}
 In the following examples, we will see that the bounds for $q$ in Lemma~\ref{lem: a bound for q for splitting bundle}, hence in the proof of Theorem~\ref{thm: elliptic ruled}, are optimal.
\end{remark}

\begin{example}\label{ex: positive degree q=degree}
Let $q\geq 2$, $\omega=(1+\sqrt{1-4q})/2$, and $C=\mathbb C/(\mathbb Z+\mathbb Z\omega)$. Then multiplication by $\omega$ gives a cyclic isogeny $\alpha:C\to C$ of degree $q$ such that $\Ker\alpha=\langle a = \frac{1}{\omega}\rangle\cong\mathbb Z/q\mathbb Z$. Let $\mathcal L=\mathcal O_C(q[0])$. Then $\deg\mathcal L=q$ and $K(\mathcal L)=\Ker\lambda_{\mathcal{L}}=C[q]$, hence $a\in K(\mathcal L)$. Choose $\theta:t_a^*\mathcal L\cong\mathcal L$ such that $\theta^q=\mathrm{id}$. Let $\xi_{q^2}$ be a primitive $q^2$-th root of unity. The isomorphism
$$
t_a^*(\mathcal O_C\oplus\mathcal L)\cong\mathcal O_C\oplus t_a^*\mathcal L\xrightarrow{\mathrm{id}\oplus\xi_{q^2}\theta}\mathcal O_C\oplus\mathcal L
$$
induces an automorphism $\tau$, over $t_a$, of $X:=\mathrm{Proj}_C(\mathcal O_C\oplus\mathcal L)$.

Then $\tau^q$ acts trivially on $C$ and fiberwise by $[U:V]\mapsto[U:\xi_q V]$, so $\tau$ has order $q^2$. Let $G=\langle\tau\rangle$ and set $Z=X/G$. For $H=\langle\tau^q\rangle$, one has $X/H\cong \mathrm{Proj}_C(\mathcal O_C\oplus\mathcal L^{\otimes q})$ and the quotient map $\rho:X\to X/H$ satisfies $\rho^*\mathcal O_{X/H}(1)\cong\mathcal O_X(q)$. The residual group $G/H\cong\Ker\alpha$ descends this bundle along $\alpha$, hence $Z\cong\mathrm{Proj}_C(\mathcal O_C\oplus\mathcal M)$ for some $\mathcal M$ with $\alpha^*\mathcal M\cong\mathcal L^{\otimes q}$. Thus $q\deg\mathcal M=\deg(\mathcal L^{\otimes q})=q^2$, so $\deg\mathcal M=q=\deg\mathcal L$; since degree $q$ line bundles on $C$ are translates of one another, $Z\cong X$.

Let $p:X\to Z$ be the quotient map and define $f$ to be the composition $$f: X\xrightarrow{p}Z\cong X.$$ Then $f$ is a cyclic self-cover with Galois group $\mathbb Z/q^2\mathbb Z$, and $f^*\mathcal O_X(1)\cong\mathcal O_X(q)$. Let $\pi:X\to C$ be the ruling and let $f|_C:C\to C$ be the map induced by $f$ on the base. Then $f|_C$ is a translation composed with $\alpha$, hence $\deg(f|_C)=q$. Since $f|_C$ is on a curve and of degree $q> 1$, it is polarized. So there is an ample line bundle $A$ on $C$ such that $(f|_C)^*A\cong A^{\otimes q}$. Then $\mathcal H=\mathcal O_X(1)\otimes\pi^*A^{\otimes N}$ is ample for some $N\gg 1$ and $f^*\mathcal H\cong\mathcal O_X(q)\otimes\pi^* (f|_C)^*A^{\otimes N}\cong\mathcal H^{\otimes q}$. Thus $f$ is a $q$-polarized cyclic cover. By Lemma~\ref{lem: facts of essential dimension}(3), $\mathrm{ed}(f)=1$.
\end{example}

\begin{example}\label{ex: torsion optimality}
Take an integer $q\ge2$. Choose $\rho$ with $\rho^2=1-q$, set $\Lambda=\mathbb Z+\mathbb Z\rho$, and put $C=\mathbb C/\Lambda$. Multiplication by $\alpha=1+\rho$ preserves $\Lambda$, hence defines a self-isogeny $\alpha:C\to C$ with $\deg\alpha=(1+\rho)(1-\rho)=q$ such that $\Ker\alpha=\langle \frac{1}{\alpha}\rangle\cong\mathbb Z/q\mathbb Z$. Let $\mathcal{L}\in \Pic^0(C)$ be a line bundle of order $q$ such that $\alpha^*\mathcal L\cong \mathcal O_C$ and set $$X :=\mathbf{P}_C(\mathcal O_C\oplus \mathcal L).$$ Now we construct a $q$-polarized cyclic self-cover of $X$.
Consider the pullback of $X$ by $\alpha$: set $X_C:=X\times_{C,\alpha} C\cong C\times\mathbf{P}^1$. We define a (polarized) morphism on $C\times\mathbf{P}^1$ by $$\bar F:C\times \mathbf{P}^1\to C\times\mathbf{P}^1, \hskip 1pc (P,[U:V])\mapsto (\alpha(P),[U^q:V^q]).$$ Then $\bar F$ descends to an endomorphism $f:X \to X$ via $C\times\mathbf{P}^1=X_C\to X$, and $f$ is polarized by \cite[Theorem~1.3]{MZ18}. It remains to prove that $f$ is a cyclic cover. Choose a primitive $q^2$-th root of unity $\xi_{q^2}$. The automorphism $$\bar \tau: (P,[U:V]) \mapsto (P+\frac{1}{\alpha^2},[U:\xi_{q^2} V])$$ on $C\times\mathbf{P}^1$ descends to an automorphism $\tau$ on $X$ such that $f\circ\tau=f$ (so $f$ is $\langle \tau \rangle$ equivariant) and $\tau^{q^2}=1$. Since 
$\frac{q}{\alpha^2}=\frac{\bar \alpha}{\alpha}\equiv\frac{2}{\alpha}\in \Ker\alpha \,({\rm mod} \Lambda),$
$\tau^q$ acts on each fiber of $X\to C$ by $[U:V] \mapsto[U:\xi_q V]$. Thus $\tau$ has order $q^2=\deg f$. So $f$ is a $q$-polarized cyclic self-cover with $\Gal(f) = \langle \tau \rangle$. By Lemma~\ref{lem: facts of essential dimension}(3), $\mathrm{ed}(f)=1$.
\end{example}


\begin{thebibliography}{1}
%
%
%
%
\bibitem{Ati57}
M. F. Atiyah, 
\newblock{\em Vector bundles over an elliptic curve.} \newblock{Proc. London Math. Soc. (3) \textbf{7} (1957), no.1, 414--452.}
\bibitem{Bea10}
A. Beauville, \newblock{\em Finite subgroups of $\mathrm{PGL}_2(K)$}. \newblock{Contemp. Math., \textbf{522} (2010), 23--29.}
\bibitem{BR97}
J.~Buhler and Z.~Reichstein,
\newblock {\em On the essential dimension of a finite group.}
\newblock {Compos. Math. \textbf{106} (1997), no. 2, 159--179.}

\bibitem{BR99}
J.~Buhler and Z.~Reichstein,
\newblock {\em On Tschirnhaus transformations.}
\newblock {Topics in Number Theory. Math. Appl. \textbf{467} (1999), 127--142.}


\bibitem{FKW24}
B.~Farb, M.~Kisin, and J.~Wolfson,
\newblock {\em Essential dimension via prismatic cohomology.}
\newblock {Duke Math. J. \textbf{173} (2024), no. 15, 3059--3106.}


\bibitem{FS22}
N. Fakhruddin and R. Saini,
\newblock {\em Finite group scheme actions and incompressibility of
Galois covers: beyond the ordinary case.}
\newblock {Doc. Math. \textbf{27} (2022), 151--182.}

\bibitem{Ful}
W.~Fulton,
\newblock {\em Intersection theory.}
\newblock {Second edition. Ergebnisse der Mathematik und ihrer Grenzgebiete. 3. Folge. A Series of Modern Surveys in Mathematics, Springer-Verlag, Berlin, (1998).}

\bibitem{GKP13} D.~Greb, S.~Kebekus, and T.~Peternell,
\newblock {\em \'Etale fundamental groups of Kawamata log terminal spaces, flat sheaves and quotients of abelian varieties}.
\newblock {Duke Math. J. \textbf{165} (2016), no. 10, 1965--2004.}

\bibitem{Har} R.~Hartshorne,
\newblock {\em Algebraic geometry}.
\newblock {Graduate Texts in Mathematics, \textbf{52} (1977).}

\bibitem{KM}
J.~Koll\'ar and S.~Mori,
\newblock{\em Birational geometry of algebraic varieties}.
\newblock{Cambridge Tracts in Math, \textbf{134} (1998).}

\bibitem{KS13}
H.~Kraft and I.~Stampfli,
\newblock {\em On Automorphisms of the Affine Cremona Group.}
\newblock {Ann. Inst. Fourier. \textbf{63} (2013), no. 3, 1137--1148.}

\bibitem{KZ}
J.~Koll\'ar and Z.~Zhuang,
\newblock{\em Essential dimension of isogenies.}
\newblock{Pure Appl. Math. Q. \textbf{22} (2026), no. 2, 675–690.}

\bibitem{Laz}
R.~Lazarsfeld,
\newblock{\em Positivity in algebraic geometry. I},
\newblock{A Series of Modern Surveys in Mathematics, \textbf{48} (2004).}

\bibitem{LOZ25}
Y.~Luo, K.~Oguiso and D.-Q.~Zhang,
\newblock{Essential dimensions of polarized endomorphisms of abelian varieties},
\newblock{preprint, arXiv:2512.04942v1.}

\bibitem{LZ26}
Y.~Luo and D.-Q.~Zhang,
\newblock{Galois self-covers of projective spaces and essential dimensions}.
\newblock{preprint, arXiv:2606.10207v1.}

\bibitem{MZ18} S.~Meng and D.-Q.~Zhang,
\newblock {\em Building blocks of polarized endomorphisms of normal projective varieties}.
\newblock {Adv. Math. \textbf{325} (2018), 243--273.}

\bibitem{Mum74}
D.~Mumford,
\newblock {\em Abelian Varieties.}
\newblock {Oxford University Press, (1974).}

\bibitem{Nak02}
N. Nakayama,
\newblock{\em Ruled surfaces with non-trivial surjective endomorphisms}.
\newblock{Kyushu J. Math. \textbf{56} (2002), no. 2, 433--446.}

\bibitem{NZ}
N. Nakayama and D.-Q. Zhang,
\newblock{\em Polarized endomorphisms of complex normal varieties}.
\newblock{Math. Ann. \textbf{346} (2010), no. 4, 991--1018.}

\bibitem{Reic10}
Z.~Reichstein,
\newblock{\em Essential dimension}.
\newblock{Proceedings of the International Congress of Mathematicians. Volume II (2010), 162--188.}

\bibitem{Sat12}
M. Satriano,
\newblock{\em The Chevalley--Shephard--Todd theorem for finite linearly reductive group schemes}.
\newblock{Algebra \& Number Theory \textbf{6} (2012), 1--26.}


\bibitem{ST54}
G. C. Shephard and J. A. Todd, 
\newblock{\em Finite unitary reflection groups}. 
\newblock{Canadian J. Math. \textbf{6} (1954), 274--304.}

\bibitem{Zh10}
D.-Q.~Zhang,
\newblock {\em Polarized endomorphisms of uniruled varieties.}
\newblock {Compos. Math. \textbf{146} (2010), 145--168.}

\bibitem{ZhS06}
S.-W.~Zhang,
\newblock{Distributions in algebraic dynamics.}
\newblock{In \emph{Surveys in differential geometry. Vol. X}, \emph{Surv. Differ. Geom.}, vol. 10 (2006), 381--430.}

\end{thebibliography}
\end{document}